\documentclass[12pt]{amsart}
\usepackage[norelsize]{algorithm2e}

\newtheorem{theorem}{Theorem}[section]

\newtheorem{corollary}[theorem]{Corollary}

\theoremstyle{definition}
\newtheorem{definition}[theorem]{Definition}

\theoremstyle{remark}

\numberwithin{equation}{section}

\def\a{\alpha}                  
\def\d{\delta}

                   \def\o{\omega}

             \def\L{\Lambda}

\newcommand{\hol}{{\mathcal H}}
\DeclareMathOperator{\og}{O}
\DeclareMathOperator{\ol}{o}

\def\D{{\mathbb D}}

\def\C{{\mathbb C}}  
 \def\R{{\mathbb R}}

\newcommand{\n}[1]{\Vert#1\Vert}

\theoremstyle{theorem}

\begin{document}

\title{Mean Lipschitz conditions on Bergman space }

\author{ P. Galanopoulos  }

\address{Department of Mathematics, University of Thessaloniki,
54124 Thessaloniki, Greece}

\email{petrosgala@math.auth.gr}

\author{ A. G. Siskakis}
\email{siskakis@math.auth.gr}

\author{ G. Stylogiannis }
\email{stylog@math.auth.gr}


\begin{abstract}
For $f$ analytic on the unit disc let $r_t(f)(z)=f(e^{it}z)$ and $f_r(z)=f(rz)$, rotations and dilations respectively. We show that for  $f$ in the Bergman space $A^p$ and $0<\alpha\leq 1$  the following  are equivalent.
\begin{itemize}
\item[(i)] $\n{r_t(f)-f}_{A^p}=\og(|t|^{\alpha}), \quad t\to 0$,
\item[(ii)] $\n{(f')_r}_{A^p}  =\og\left (1-r)^{\alpha-1}\right ), \quad r\to 1^{-}$,
\item[(iii)] $\n{f_r-f}_{A^p}=\og((1-r)^{\alpha}),\quad  r\to 1^{-}$.
\end{itemize}

The  Hardy space analogues of these conditions are known to be equivalent by results of Hardy and Littlewood and of  E. Storozhenko, and in that   setting they   describe the
mean Lipschitz spaces $\Lambda (p, \alpha)$.

On the way,  we provide an elementary proof of the equivalence of $(ii)$ and $(iii)$ in Hardy spaces, and show that
similar  assertions are valid  for certain weighted mean Lipschitz spaces.
\end{abstract}

\maketitle

\section{Introduction}

Let $\D$ denote the unit disc in the complex plane $\C$. For $1\leq p\leq \infty$ and  $f:\D\to \C$ analytic the
integral mean $M_p(r, f)$, $0\leq r<1$,  is  defined
$$
M_p(r,f) =\left(\frac{1}{2\pi}\int_{-\pi}^{\pi}|f(re^{i\theta})|^p\,d\theta\right)^{1/p},
$$
and
$$
M_{\infty}(r,f)=\max_{-\pi\leq \theta <\pi}|f(re^{i\theta})|.
$$
$M_p(r,f)$  is  an increasing function of $r$. The Hardy space $H^p$ consists of all $f\in \hol(\D)$ for which
$$
\n{f}_p=\sup_{r<1}M_p(r,f)=\lim_{r\to 1^{-}}M_p(r,f)<\infty.
$$
 Each $f$ in $H^p$ has radial limits
$$
f^{*}(\theta)=\lim_{r\to 1^{-}}f(re^{i\theta})
$$
almost everywhere on $\theta \in [-\pi, \pi]$.  The so defined boundary function $f^{*}$ is $p$-integrable (essentially bounded if $p=\infty$) and $\n{f}_p$ can  be recovered from $f^{*}$  as
$\n{f}_p=\left(\frac{1}{2\pi}\int_{-\pi}^{\pi}|f^{*}(\theta)|^p d\theta\right)^{1/p}$.

For $1\leq p \leq \infty$, $H^p$ are Banach spaces. The linear map $f\to f^{*}$ identifies $H^p$ as the  closed subspace of $L^p(-\pi, \pi)$ generated by the set of exponentials $\{e^{in\theta}\}_{n=0}^{\infty}$ or equivalently the subspace consisting of all functions of $L^p(-\pi, \pi)$ whose Fourier series is \lq\lq of power series type\rq\rq. Additional information for Hardy spaces   can be found in  \cite{Du}, and we follow this reference for notation and related material.

For $f$ analytic on $\D$ we write
$$
r_t(f)(z) =f(e^{it}z), \quad t\in \R,
$$
for the rotated function. If $f$ has boundary values $f^{\ast}$ a.e. on
$[-\pi, \pi]$ we view $f^{\ast}$ extended periodically and we  write
$$
\tau_t(f^{\ast})(\theta)=f^{*}(\theta+t), \quad t\in \R
$$
for  the translated  $f^{*}$.  It is clear that in this case $r_t(f)^{\ast}=\tau_t(f^{\ast})$.

For $f\in H^p$, $p<\infty$,  by the continuity of the integral  we have $\lim_{t\to 0}\n{\tau_t(f^{*})-f^{*}}_p=0$. Specifying the rate of this convergence imposes restriction on $f$.

\begin{definition}For $1\leq p < \infty$ and $0<\a\leq 1$ the analytic mean Lipschitz space $\Lambda\left(p,\alpha\right)$ is the collection of $f\in H^p$ such that
\begin{equation}{\label{Lip-def}}
\n{\tau_t(f^{*})-f^{*}}_p\leq C|t|^{\a},  \quad  -\pi\leq t< \pi,
\end{equation}
where $C$ is a constant.  The subspace $\lambda\left(p, \a\right)$ consists of all $f\in H^p$ which satisfy
\begin{equation}{\label{lip-def}}
\n{\tau_t(f^{*})-f^{*}}_p =\ol(|t|^{\a}),  \quad  t\to 0.
\end{equation}
\end{definition}

Note that these spaces can be defined more generally, as subspaces of $L^p(-\pi, \pi)$, to consist of all  $L^p$ functions $f$ that satisfy  (\ref{Lip-def}) and  (\ref{lip-def}) (with $f$ in place of $f^{*}$).
 It was in this  general setting that they were first studied in the  1920's by Hardy, Littlewood and others, in connection with convergence and summability of Fourier series, fractional integrals and fractional derivatives, see \cite{HL1}, \cite{HL2}, \cite{HL3}. Among several other results Hardy and Littlewood proved the following theorem.

\vspace{0,15cm}
\textbf{Theorem A}.  \textit{Suppose $1\leq p<\infty$,  $0<\a\leq 1$ and $f\in H^p$. Then the following are
equivalent
\begin{itemize}
\item[(a)] $f\in  \Lambda\left(p,\a\right)$,
\item[(b)] $M_p(r,f^\prime) =\og\left ((1-r)^{\alpha -1}\right ), \quad r\to 1^{-}$.
\end{itemize}}

\vspace{0,15cm}
Note that if an analytic function  satisfies (b) for some $\a>0$, then it belongs in $H^p$.

For $f$ analytic on $\D$ let
$$
f_r(z)=f(rz), \quad 0\leq r<1,
$$
be  the dilations of $f$. Each $f_r$ is   analytic on a disc of radius $1/r>1$  and if $f\in H^p$ then
$$
\n{f_r-f}_p \to 0, \quad \text{as}\,\,\, r\to 1^{-},
$$
(\cite[Theorem 2.6]{Du}). In  classical terminology this states  that the Fourier series of $f^{*}$ (the power series of $f$ on the boundary) is Abel summable in the $L^p$ norm.

It seems to be not so well known that for $f\in H^p$,  membership of $f$ in  $\Lambda(p, \a)$ is equivalent to a condition on the rate at which $\n{f_r-f}_p\to 0$.

\vspace{0,15cm}
\textbf{Theorem B}(\cite{ES}). \textit{ Let $1\leq p<\infty$, $0<\a\leq 1$ and $f\in H^p$. Then the following are
equivalent
\begin{itemize}
\item[(a)] $f\in  \Lambda\left(p,\a\right)$,
\item[(c)] $\n{f_r-f}_p=\og((1-r)^{\a}),\quad r\to 1^{-}$.
\end{itemize}}

\vspace{0,15cm}
This theorem may be deduced from  the work  \cite{ES} of E. A.  Storozhenko.  In that article the author studies the classical $L^p$ modulus of continuity of order $1$ and of higher orders for Hardy functions, and introduces a modified modulus of continuity which at order $1$ coincides with the classical one. To indicate how Theorem B follows we  briefly summarize  the relevant information. Recall that the $L^p$ modulus of continuity  $\o_p(\d, f)$ (of order $1$) of a function $f\in L^p[-\pi, \pi]$ extended periodically, is
$$
\omega_p(\delta, f)=\sup_{|t|\leq \delta}\n{\tau_t(f)-f}_p.
$$
It is easy to see that    $\Lambda\left(p,\a\right)$ can  equivalently be defined as the space of all $f\in H^p$ for which $\o_p(\d, f^{*})\leq C\d^{\a}$ for all small positive $\d$. From \cite[Theorem 1 and Theorem 6]{ES}  it follows  that if $f\in H^p$ then there are constants $C_1(p), C_2(p)$ such that
$$
C_1 \o_p(1-r, f^{*})\leq \n{f_r-f}_p\leq C_2\o_p(1-r, f^{*}), \quad  r\to 1^{-},
$$
 and from this Theorem B follows.  Storozhenko  attributes  special cases
of the result to previous work of  A. I. Buadze and of R. M. Trigub, see \cite{ES} for details. Other   special cases,
 explicit or implicit, can be found in \cite{Wa},  \cite{HS}, \cite{JP},  \cite{Pa3}.

Note  that $M_p(r, f)=\n{f_r}_p$, so condition (b) of Theorem A  can be restated as
$$
\n{(f')_r}_p=\og\left ((1-r)^{\alpha -1}\right ),
$$
and the theorem of Hardy and Littlewood may  be interpreted as saying that smoothness of the boundary function $f^{*}$ can be detected from inside the disc in terms of the dilations of $f'$. Theorem B says that in effect
that dilations of $f$ itself can be used for detecting smoothness of $f^{*}$. To stress this point of view but also to have handy  the statement we  prove below,  we  formulate

\newpage
\textbf{Theorem C}. \textit{ Let $1\leq p<\infty$, $0<\a\leq 1$ and $f\in H^p$. Then the following are
equivalent
\begin{itemize}
\item[(b)] $M_p(r,f^\prime) =\og\left ((1-r)^{\alpha -1}\right ), \quad r\to 1^{-}$
\item[(c)] $\n{f_r-f}_p=\og((1-r)^{\a}), \quad r\to 1^{-}$.
\end{itemize}}

\vspace{0,15cm}
The proofs of Theorem B or of its special cases in the above mentioned articles  are  technically complicated and use  rather advanced techniques.   We are going to present an elementary proof of Theorem C below. We then study the analogous equivalence of conditions (a), (b) and (c) on weighted mean Lipschitz spaces,  on Bergman spaces and on the  Dirichlet space.

\section{An elementary proof of  Theorem C}

In what follows the letters $C, C', ...$ will denote  constants, whose value may change  at each step.

To prove  that $(b) \Rightarrow (c)$ we will use the   identity,
\begin{equation}\label{Ef1}
f(z)-f_r(z)=\int_r^1zf'(sz)\,ds,
\end{equation}
which is valid for all $f$ analytic on $\D$, all $z\in \D$ and $r<1$.

We take  p-integral means on the circle of radius $u\in (0,1)$ to obtain,
\begin{align*}
M_p(u, f_r-f) &= \left(\frac{1}{2\pi}\int_0^{2\pi}\left|\int_r^1u e^{i\theta}f'(su e^{i\theta})\,ds\right|^p\,d\theta\right)^{1/p}\\
& \leq \int_r^1\left(\frac{1}{2\pi}\int_0^{2\pi}|u e^{i\theta}f'(su e^{i\theta})|^pd\theta\right)^{1/p}\,ds\\
& =u\int_r^1 M_p(su, f')\,ds.
\end{align*}
If in addition $f\in H^p$  take the supremum on $u<1$,
\begin{equation}\label{E1}
\n{f_r-f}_p \leq \int_r^1 M_p(s, f')\,ds,
\end{equation}
and if $f$ satisfies (b) then,
$$
\n{f_r-f}_p  \leq   C\int_r^1(1-s)^{\a-1}\,ds = \frac{C}{\a}(1-r)^{\a},
$$
as desired.

For the converse suppose,  $f$ is analytic on $\D$,   $0<r<1$ and   $z\in \D$, then
\begin{align*}
z\int_r^1(f'(sz)-f'(z))\,ds &=\int_r^1\frac{d}{ds}(f(sz))\,ds - zf'(z)(1-r)\\
& = f(z)-f(rz)-zf'(z)(1-r),
\end{align*}
therefore
\begin{equation}\label{Ef2}
(1-r)f'(z) = \frac{f(z)-f_r(z)}{z}+\int_r^1(f'(z)-f'(sz))\,ds,
\end{equation}
and note that $\frac{f(z)-f_r(z)}{z}$ has removable singularity at $0$.

Taking integral means of both sides on the circle $|z|=u$ and using Minkowski's inequality we have
\begin{align*}
(1-r)M_p(u, f') & \leq \left(\frac{1}{2\pi}\int_0^{2\pi}\frac{|f(ue^{i\theta})-f_r(ue^{i\theta})|^p}{u^p}\,d\theta\right)^{1/p}\\
&+\left(\frac{1}{2\pi}\int_0^{2\pi}\left| \int_r^1 (f'(ue^{i\theta})-f'(sue^{i\theta}))\,ds\right|^p\,d\theta\right)^{1/p}\\
&\leq \frac{1}{u}M_p(u, f_r-f)+\int_r^1 \left(\frac{1}{2\pi}\int_0^{2\pi}  |f'(ue^{i\theta})-f'(sue^{i\theta})|^p\,d\theta \right)^{1/p}\,ds.
\end{align*}
Thus
\begin{equation}\label{Ef3}
(1-r)M_p(u, f') \leq\frac{1}{u} M_p(u, f_r-f) + \int_r^1M_p(u, \Phi_{[s]}')\,ds
\end{equation}
where $\Phi_{[s]}(z)=f(z)-\frac{1}{s}f_s(z)$ for $r\leq s<1$.

We estimate the two terms in the right hand side of (\ref{Ef3}).
The quantity  $\frac{1}{u}M_p(u, f_r-f)$ is increasing in $u$ since it is the  integral mean of
an analytic function. Assuming  that $f\in H^p$ we have for each $0< u<1$,
\begin{equation}\label{Ef4}
\frac{1}{u}M_p(u, f_r-f)\leq \sup_{u<1}\frac{1}{u}M_p(u, f_r-f)=\n{f_r-f}_p.
\end{equation}
Next if  $f\in H^p$ then $\Phi_{[s]}\in H^p$ and
\begin{align*}
\n{\Phi_{[s]}}_p &= \n{f-\frac{1}{s}f_s}_p
 \leq \n{f-f_s}+\n{f_s-\frac{1}{s}f_s}_p\\
& =\n{f_s-f}_p+\frac{1-s}{s}\n{f_s}_p\\
&\leq \frac{1}{r}(1-s)\n{f}_p+\n{f_s-f}_p
\end{align*}
for $r\leq s<1$.
Further we use the  well known estimate
$$
M_p(u, F')\leq \frac{C\n{F}_p}{1-u}, \quad 0\leq u<1,
$$
for functions $F\in H^p$, with  $C$ independent of $F$.  For a  proof of this one can use  Cauchy's integral  formula for $F'(z)$ as in the proof of the first  part of \cite[Theorem 5.5]{Du}.

Applying this inequality to  $\Phi_{[s]}$ together with the  inequality for  $\n{\Phi_{[s]}}_p$  we have
\begin{align*}
(1-u)M_p(u, \Phi_{[s]}')\leq \n{\Phi_{[s]}}_p\leq \frac{C}{r}(1-s)\n{f}_p+C\n{f_s-f}_p,
\end{align*}
valid for $0\leq u<1$ and $r\leq s<1$. In particular letting $u=r$ and integrating we obtain
$$
\int_r^1M_p(r, \Phi_{[s]}')\,ds\leq \frac{C\n{f}_p}{2r}(1-r)+ \frac{C}{1-r}\int_r^1\n{f_s-f}_p\,ds.
$$
Using this  and letting $u=r$ in (\ref{Ef3}), together with  (\ref{Ef4})  we find
\begin{equation}\label{E2}
\begin{aligned}
(1-r)M_p(r, f')&\leq \n{f_r-f}_p+\frac{C\n{f}_p}{2r}(1-r)\\
&+ \frac{C}{1-r}\int_r^1\n{f_s-f}_p\,ds.
\end{aligned}
\end{equation}
 Thus if $f$ satisfies  $\n{f_r-f}_p\leq C'(1-r)^{\a}$ then we have
$$
(1-r)M_p(r, f')\leq C' (1-r)^{\a}+\frac{C\n{f}_p}{2r}(1-r)+\frac{C C'}{\a+1}(1-r)^{\a}
$$
therefore
$$
M_p(r, f')=\og((1-r)^{\a-1}), \quad r\to 1^{-},
$$
valid for all $0<\a\leq 1$. This finishes the proof of Theorem  C.\\

It is well known that the \lq\lq little oh\rq\rq analogue of the theorem of Hardy and Littlewood is valid for membership in the spaces $\lambda(p, \a)$. That is, $f\in \lambda(p, \a)$ if and only if
$M_p(r,f^\prime) =\ol\left ((1-r)^{\alpha -1}\right )$, as $r\to 1^{-}$. As expected the \lq\lq little oh\rq\rq analogue  of Theorem C also holds.  The  proof follows easily from the inequalities (\ref{E1}) and (\ref{E2}). We omit the details.

\begin{corollary}\label{cor1}
Let $1\leq p<\infty$, $0<\a\leq 1$ and $f\in H^p$. Then the following are
equivalent
\begin{itemize}
\item[(b)] $M_p(r,f^\prime) =\ol\left ((1-r)^{\alpha -1}\right ), \quad r\to 1^{-}$,
\item[(c)] $\n{f_r-f}_p=\ol((1-r)^{\a}), \quad r\to 1^{-}$.
\end{itemize}
\end{corollary}

\section{Weighted mean Lipschitz  Spaces}

We  show that the analogue of Theorem C remains valid in  weighted Lipschitz spaces which are defined as follows.

Let $\o:[0, 1)\to [0, \infty)$
be a continuous and nondecreasing  function with $\o(0)=0$.
The weighted mean Lipschitz space
$\Lambda(p, \o)$ consists of all $f\in H^p$ such that
\begin{equation}\label{Lip-w}
\n{\tau_t(f^{*})-f^{*}}_p=\og\left(\o(t)\right), \quad t\to 0.
\end{equation}
If  $\o(t)=t^{\a}$ then we recover $\L(p,\a)$. For general weights $\o$ these spaces  have been studied by various authors, see for example \cite{BS}, \cite{Gi1}, \cite{GiG} and the references therein.

An extension of the theorem of Hardy and Littlewood was proved for  these generalized Lipschitz spaces for an appropriate class of weights $\o$.
 A weight $\o$ is called a   \textit{Dini weight} if  $\o(t)/t$ is integrable on $[0,1)$ and there is a constant $C$ such that the  following condition is satisfied
\begin{equation}\label{Dini}
\int_0^t\frac{\o(s)}{s}\,ds\leq C\o(t), \quad 0<t<1.
\end{equation}
A weight $\o$ is an \textit{admissible weight} if it is a Dini weight and satisfies in addition  the    condition
\begin{equation}\label{Dini+b}
\int_t^1\frac{\o(s)}{s^2}\,ds\leq C\frac{\o(t)}{t}, \quad 0<t<1,
\end{equation}
for a constant $C$. Note that if $\o$  satisfies this last condition then there is a constant $C>0$ such that
\begin{equation}\label{Dini+bp}
\frac{\o(t)}{t}\geq C, \quad 0<t<1.
\end{equation}

For admissible weights  O. Blasco and G. S. de Souza proved in
 \cite[Theorem 2.1]{BS}  the analogue of the theorem of Hardy and Littlewood.

\vspace{0,15cm}
\textbf{Theorem D} (\cite{BS}) \textit{Suppose  $1\leq p < \infty$, $\o$ is an admissible weight and $f$ is analytic on $\D$.  Then the following are equivalent
\begin{itemize}
\item[(a)] $f\in \Lambda\left(p,\o\right)$,
\item[(b)] $M_p(r,f^\prime) =\og\left (\frac{\o(1-r)}{1-r}\right ), \quad r\to 1^{-}$.
\end{itemize}}

\vspace{0,15cm}

We show that the analogue of Theorem C is valid in $\Lambda\left(p,\o\right)$.

\begin{theorem} Suppose $1\leq p<\infty$, $\o$ is a Dini weight that satisfies  (\ref{Dini+bp})  and $f\in H^p$. Then the following are equivalent
\begin{itemize}
\item[(b)] $M_p(r,f^\prime) =\og\left (\frac{\o(1-r)}{1-r}\right ), \quad r\to 1^{-}$,
\item[(c)]  $\n{f_r-f}_p=\og(\omega(1-r)), \quad r\to 1^{-}$.
\end{itemize}
In particular the conditions (a), (b) and (c) as above are equivalent for admissible weights.
\end{theorem}

\begin{proof}
Suppose  (b) holds for $f\in H^p$. Using (\ref{E1}) we  have
\begin{align*}
\n{f_r-f}_p &\leq  C\int_r^1\frac{\o(1-s)}{1-s}\,ds\\
&=C\int_0^{1-r}\frac{\o(u)}{u}\,du\\
&\leq CC'\o(1-r)
\end{align*}
as desired, where in the last step we have used the Dini property of $\o$.

Suppose now (c) holds for $f\in H^p$. Then the last term in the right hand side of (\ref{E2}) becomes
\begin{align*}
\frac{C}{1-r}\int_r^1\n{f_s-f}_p\,ds&\leq  \frac{C'}{1-r}\int_r^1\o(1-s)\,ds\\
&\leq C'\int_r^1\frac{\o(1-s)}{1-s}\,ds\\
&= C'\int_0^{1-r}\frac{\o(u)}{u}\,du\\
&\leq C''\o(1-r).
\end{align*}
From this and (\ref{E2}) we have
\begin{align*}
(1-r)M_p(r, f')& \leq C\o(1-r)+C'(1-r)+C''\o(1-r)\\
&\leq C'''\o(1-r)+C'(1-r),
\end{align*}
so that
$$
M_p(r, f')\leq C\frac{\o(1-r)}{1-r}+C', \quad \text{as}\,\,r\to 1^{-}.
$$
Taking into account (\ref{Dini+bp}) we find $M_p(r, f')\leq C''\frac{\o(1-r)}{1-r}$ as $r\to 1^{-}$ and this finishes the proof.
\end{proof}

\section{Bergman spaces and the Dirichlet  space.}

\textbf{Bergman spaces}. Let $dm(z)=\frac{1}{\pi}rd\theta dr$ be  the normalized Lebesgue area   measure on $\D$.
Recall that for $1\leq p<\infty$ the Bergman space $A^p$ consists of the   analytic functions $f:\D\to \C$ such that
$$
\n{f}_{A^p}^p=\int_{\D}|f(z)|^p\,dm(z)= 2\int_0^1M_p^p(u, f)u\,du <\infty.
$$

For $1\leq p<\infty$  $A^p$ are  Banach spaces and $H^p\subset A^p$, but in contrast to Hardy spaces  $A^p$ contains functions which do not   have boundary  radial limits.  Observe however that  if  $g$ belongs to a Hardy space  $H^p$ with boundary function  $g^{*}$ then
$$
\n{\tau_t(g^{*})-g^{*}}_p=\n{r_t(g)-g}_p.
$$
 We may therefore use  the quantity $\n{r_t(f)-f}_{A^p}$ as a substitute for measuring \lq\lq smoothness\rq\rq of functions in $A^p$.

Let $f$ be analytic on $\D$,  $1\leq p<\infty$ and $0\leq r<1$. Then the  dilations  $f_r(z)=f(rz)$ belong to $A^p$, and we define the quantity
$$
 A_p(r, f):=\n{f_r}_{A^p}= \left(\int_{\D}|f(rz)|^p\,dm(z)\right)^{1/p}.
$$
Note that $A_p(r, f)$ is an area integral mean, in   much the same way as  $M_p(r, f)=\n{f_r}_{H^p}$ is arc-length integral mean
\begin{align*}
A_p(r, f)&=\left(\frac{2}{r^2}\int_0^rM_p^p(u, f)u\,du\right)^{1/p}\\
&=\left(\frac{1}{m(r\D)}\int_{r\D}|f(z)|^p\,dm(z)\right)^{1/p}.
\end{align*}
Further,  since $M_p(r,f)$ is increasing in $r$ it follows that $A_p(r, f)$ is also an increasing function of $r$, and if $f\in A^p$ we have
$$
A_p(r, f)=\n{f_r}_{A^p}\to \n{f}_{A^p}, \quad r\to 1^{-},
$$
while if $f$ is not in $A^p$ then $A_p(r, f)$ increases to infinity. An application of the Lebesgue dominated convergence theorem shows that for $f\in A^p$,
$$
\n{f_r-f}_{A^p}\to 0, \quad r\to 1^{-}.
$$

Such  area integral means   were studied in the  more general context of volume integral means by J. Xiao and K. Zhu in \cite{XZ}, where  the authors prove basic properties of such averages for   weighted volume measures
$dv_a(z)=(1-|z|^2)^a dv(z)$ on the unit ball of $\C^n$.

We show now  that functions $f\in A^p$ for which $\n{r_t(f)-f}_{A^p}=\og(|t|^{\a})$  as $t\to 0$ can be characterized by the growth of $A_p(r, f')$ as well as by the limiting behavior of $\n{f_r-f}_{A^p}$.

\begin{theorem}\label{theorem3}
 Let $1\leq p<\infty$, $0<\a\leq 1$ and $f\in A^p$. Then the following are equivalent
\begin{itemize}
\item[(a)] $\n{r_t(f)-f}_{A^p}=\og(|t|^{\a}), \quad t\to 0$,
\item[(b)] $A_p(r, f')  =\og\left (1-r)^{\a-1}\right ), \quad r\to 1^{-}$,
\item[(c)] $\n{f_r-f}_{A^p}=\og((1-r)^{\a}),\quad  r\to 1^{-}$.
\end{itemize}
\end{theorem}

\begin{proof}  We first show that (a) and (b) are equivalent. To show (b) implies (a) we
 follow the argument in  \cite[Theorem 5.4]{Du} adapted for area integrals. Note that if (b) holds for an analytic function $f$ then $f\in A^p$.
 Indeed assuming without loss of generality that $f(0)=0$ we have
 $$
 f(z)=\int_0^zf'(\zeta)\,d\zeta=\int_0^1f'(tz)z\,dt,
 $$
and an application of Minkowski's inequality gives
\begin{align*}
\left(\int_{\D}|f(z)|^p\,dm(z)\right)^{1/p}&=\left(\int_{\D}\left|\int_0^1f'(tz)z\,dt\right|^p\,dm(z)\right)^{1/p}\\
&\leq \int_0^1\left(\int_{\D}|f'(tz)|^p\,dm(z)\right)^{1/p}\,dt\\
&=\int_0^1A_p(t, f')\,dt\\
&\leq C\int_0^1(1-t)^{\a-1}\,dt\\
&=\frac{C}{\a}.
\end{align*}

Suppose $f$ is analytic on $\D$ and satisfies (b). Let $z\in \D$,  $0<\d<1$ and $t>0$, then
\begin{align*}
f(e^{it}z)-f(z)&=\int_z^{\d z}f'(\zeta)\, d\zeta+\int_{\d z}^{\d e^{it}z}f'(\zeta)\, d\zeta+\int_{\d e^{it}z}^{ e^{it}z }f'(\zeta)\, d\zeta\\
&= -\int_{\d}^1 f'(sz)z\,ds+\int_0^tf'(\d e^{is}z)\d ze^{is}i\,ds+\int_{\d}^1f'(se^{it}z)e^{it}z\,ds\\
&=f_1(z)+f_2(z)+f_3(z).
\end{align*}
Each  function $f_i(z)$ is analytic on $\D$.  We estimate the Bergman norm of each by applying
Minkowski's inequality and by using the fact that compositions of Bergman functions with rotations leave their norm invariant. For $f_1$  we have
\begin{align*}
\n{f_1}_{A^p}&=\left(\int_{\D}\left|-\int_{\d}^1f'(sz)z\,ds\right|^p\,dm(z)\right)^{1/p}\\
&\leq \int_{\d}^1\left(\int_{\D}|f'(sz)z|^p\,dm(z)\right)^{1/p}\,ds\\
&\leq \int_{\d}^1A_p(s, f')\,ds,
\end{align*}
and similarly we find
$$
\n{f_3}_{A^p}\leq \int_{\d}^1A_p(s, f')\,ds, \quad \,\,\,\,\n{f_2}_{A^p}\leq \int_0^tA_p(\d, f')\,ds=tA_p(\d, f').
$$
 Thus
$$
\n{r_t(f)-f}_{A^p} \leq \n{f_1}_{A^p}+\n{f_2}_{A^p}+\n{f_3}_{A^p}
\leq 2\int_{\d}^1A_p(s, f')\,ds+tA_p(\d, f'),
$$
and since  by hypothesis $A_p(s, f')\leq C(1-s)^{\a-1}$,  we find
$$
\n{r_t(f)-f}_{A^p} \leq \frac{2C}{\a}(1-\d)^{\a}+Ct(1-\d)^{\a-1}.
$$
This holds with $\d$ and $t$ independent of each other. Thus given $t$ with $0<t<1$ choose $\d=1-t$. The result is
$$
\n{r_t(f)-f}_{A^p} \leq C(\frac{2}{\a}+1)t^{\a},
$$
with the constant independent of $t>0$.

Finally if  $t<0$   then  $\n{r_t(f)-f}_{A^p}=\n{f-r_{-t}(f)}_{A^p}$ with $-t>0$ and the assertion follows. This completes the proof
of the direction (b) $\Rightarrow$ (a).

We now show that (a) implies (b). Suppose $f\in A^p$ satisfies (a). Let $0<u<1$ and use the Cauchy formula to write $f'(uz)$ as line integral over the contour
$\gamma(t)=ze^{it}, -\pi\leq t\leq \pi$,
\begin{align*}
f'(uz)&=\frac{1}{2\pi i}\int_{\gamma}\frac{f(\zeta)-f(z)}{(\zeta-uz)^2}\,d\zeta
=\frac{1}{2\pi}\int_{-\pi}^{\pi}\frac{f(ze^{it})-f(z)}{(ze^{it}-uz)^2}ze^{it}\,dt\\
&=\frac{1}{2\pi}\int_{-\pi}^{\pi}\frac{f(ze^{it})-f(z)}{z}\frac{e^{it}}{(e^{it}-u)^2}\,dt.
\end{align*}
From  Minkowski's inequality  we have
\begin{align*}
A_p(u, f')&=\left(\int_{\D}|f'(uz)|^p\,dm(z)\right)^{1/p}\\
&=\left(\int_{\D}\left| \frac{1}{2\pi}\int_{-\pi}^{\pi}\frac{f(ze^{it})-f(z)}{z}\frac{e^{it}}{(e^{it}-u)^2}\,dt\right|^p\,dm(z)\right)^{1/p}\\
&\leq \frac{1}{2\pi}\int_{-\pi}^{\pi}\left(\int_{\D}\left|\frac{f(ze^{it})-f(z)}{z}\right|^p\,dm(z)
\right)^{1/p}\frac{1}{|e^{it}-u|^2}\,dt\\
&=\frac{1}{2\pi}\int_{-\pi}^{\pi} \n{\frac{f(ze^{it})-f(z)}{z}}_{A^p}
\frac{1}{|e^{it}-u|^2}\,dt.
\end{align*}
At this point we need the fact  that if $g\in A^p$ and $g(0)=0$ then there is a constant $C=C_p$
such that
$$
\n{\frac{g(z)}{z}}_{A^p}\leq C_p\n{g}_{A^p},
$$
see \cite[Lemma 4.26]{Zhu}. Using this and the hypothesis (a) we have
\begin{align*}
A_p(u, f')&\leq \frac{C_p}{2\pi}\int_{-\pi}^{\pi} \n{f(ze^{it})-f(z)}_{A^p}
\frac{1}{|e^{it}-u|^2}\,dt\\
&\leq \frac{C_pC}{2\pi}\int_{-\pi}^{\pi}\frac{|t|^{\a}}{|e^{it}-u|^2}\,dt\\
&=\frac{C_pC}{\pi}\int_0^{\pi}\frac{t^{\a}}{1-2u\cos(t)+u^2}\,dt.
\end{align*}
Now use the standard inequality
$$
1-2u\cos(t)+u^2=(1-u)^2+4u\sin^2(\frac{t}{2})\geq (1-u)^2+\frac{4ut^2}{\pi^2}
$$
which is valid for $0\leq t\leq \pi$. Making the change of variable $s=\frac{2\sqrt{u}}{(1-u)\pi}t$  and assuming $u\geq 1/4$ we obtain
\begin{align*}
\int_0^{\pi}\frac{t^{\a}}{1-2u\cos(t)+u^2}\,dt&\leq\int_0^{\pi}\frac{t^{\a}}{(1-u)^2+\frac{4ut^2}{\pi^2}}\,dt
\leq C_{\a}(1-u)^{\a-1},
\end{align*}
 where $C_{\a}= \pi^{\a+1}\int_0^{\infty}\frac{s^{\a}}{1+s^2}\,ds$ is a  constant, finite for $0<\a<1$, which does not depend on $u\in (1/4, 1)$. Thus if
 $0<\a<1$ then  $A_p(u, f')=\og((1-u)^{\a-1})$ as  $u\to 1^{-}$, the desired conclusion.

 It remains to treat the case $\a=1$. Here we have to prove that if $f\in A^p$ and $\n{r_t(f)-f}_{A^p}\leq C t$
as $t\to 0$ then $f'\in A^p$. To do this let $z\in\D$ and observe that
$$
\lim_{t\to 0}\frac{f(e^{it}z)-f(z)}{t}=  z\lim_{t\to 0}\frac{f(e^{it}z)-f(z)}{e^{it}z-z}\frac{e^{it}-1}{t}=                     izf'(z),
$$
i.e. $\frac{f(e^{it}z)-f(z)}{t}$ converges pointwise to $izf'(z)$ on $\D$ as $t\to 0$. By the assumption we also have
$\n{\frac{r_t(f)-f}{t}}_{A^p}\leq C$. Thus by Fatou's Lemma,
\begin{align*}
\int_{\D}|zf'(z)|^p\,dm(z)&=\int_{\D}\liminf_{t\to 0}\left|\frac{f(e^{it}z)-f(z)}{t}\right|^p\,dm(z)\\
&\leq \liminf_{t\to 0}\int_{\D}\left|\frac{f(e^{it}z)-f(z)}{t}\right|^p\,dm(z)\\
&\leq C^p,
\end{align*}
so that $zf'(z)$ and therefore $f'(z)$ are in $A^p$.

Next we show that (b) and (c) are equivalent.
 Assume $f$ is analytic  on $\D$ and satisfies (b). Fix $r\in (0,1)$ and take integral means on $|z|=u\in (0, 1)$ in the equation (\ref{Ef1}) to obtain
$$
M_p(u, f_r-f) \leq u\int_r^1 M_p(su, f')\,ds\leq \int_r^1 M_p(su, f')\,ds .
$$
Recalling that  (b) implies that $f\in A^p$ and we find
\begin{align*}
\n{f_r-f}_{A^p} &= \left(2\int_0^1M_p^p(u, f_r-f)u\,du\right)^{1/p}\\
&\leq \left( 2\int_0^1\left(\int_r^1M_p(su, f')\,ds\right)^p u\,du\right)^{1/p}\\
&\leq \int_r^1\left(2\int_0^1M_p^p(su, f')u\,du\right)^{1/p}\,ds\\
&=\int_r^1\left(\frac{2}{s^2}\int_0^sM_p^p(v, f')v\,dv\right)^{1/p}\,ds\\
&=\int_r^1A_p(s, f')\,ds.
\end{align*}
Since   $A_p(s, f')\leq C(1-s)^{\a-1}$  the  integration gives
$\n{f_r-f}_{A^p}\leq C'(1-r)^{\a}$ so (c) holds.

Finally we prove (c) implies (b). Observe first that if $f, g$ are analytic on $\D$ then since  $\n{(f+g)_r}_{A^p}\leq \n{f_r}_{A^p}+\n{g_r}_{A^p}$ we will have
$$
A_p(r, f+g)\leq A_p(r, f)+A_p(r, g)
$$
for each $0\leq r<1$.

Assume $f\in A^p$ and fix $r\in (0,1)$. Taking area integral means in both sides of  (\ref{Ef2}) we find that for each $0<u<1$
\begin{equation}\label{Be1}
(1-r)A_p(u, f')\leq A_p(u, F_{[r]}) +A_p(u, \Psi_{[r,s]}),
\end{equation}
where $F_{[r]}(z)=\frac{f(z)-f_r(z)}{z}$ and
$$
\Psi_{[r,s]}(z)=\int_r^1(f'(z)-f'(sz))ds.
$$
For  the term $A_p(u, F_{[r]})$ in (\ref{Be1}) we obtain
\begin{equation}\label{Be2}
A_p(u, F_{[r]})=\n{(F_{[r]})_u}_{A^p}\leq \n{F_{[r]}}_{A^p}\leq C_p\n{f-f_r}_{A^p}
\end{equation}
for each $0<u<1$.
Now we look at the the term  $A_p(u, \Psi_{[r,s]})$ and use Minkowski's inequality,
\begin{align*}
A_p(u, \Psi_{[r,s]})&= \n{\Psi_{[r,s]}(uz)}_{A^p}\\
&= \left(\int_0^1\int_0^{2\pi}\left|\int_r^1(f'(vue^{i\theta})-f'(svue^{i\theta}))\,ds\right|^p\, \frac{d\theta}{\pi}v\,dv\right)^{1/p}\\
&\leq \left(\int_0^1\left(\int_r^1 2^{1/p}M_p(uv, \Phi_{[s]}')\,ds\right)^pv\, dv\right)^{1/p}\\
\intertext{where $\Phi_{[s]}(z)=f(z)-\frac{1}{s}f_s(z)$ for $r\leq s<1$, and further }
&\leq \int_r^1 \left(2\int_0^1M_p^p(uv, \Phi_{[s]}')v\,dv\right)^{1/p}\,ds\\
&= \int_r^1\left(2\int_0^1M_p^p(v, (\Phi_{[s]}')_u)v\,dv\right)^{1/p}\,ds\\
&=\int_r^1A_p(u, \Phi_{[s]}')\,ds.
\end{align*}
At this point we need the  inequality
\begin{equation}\label{Be3}
A_p(u, F')\leq \frac{C\n{F}_{A^p}}{1-u}, \quad 0\leq u<1,
\end{equation}
for $F\in A^p$, where $C$ is a constant independent of $F$. A proof of this is as follows.
For  $F\in A^p$ the  Bergman  norm of $\n{F}_{A^p}^p$ is equivalent to the quantity
$$
|F(0)|^p+ \int_{\D}|F'(z)|^p(1-|z|^2)^p\,dm(z),
$$
see \cite[page 85]{Zhu}. Thus there is a constant $C$ such that for $0<u<1$,
\begin{align*}
\n{F}_{A^p}^p&\geq C\int_{\D}|F'(z)|^p(1-|z|^2)^p\,dm(z)\\
&\geq C\int_{u\D}|F'(z)|^p(1-|z|^2)^p\,dm(z)\\
&\geq C(1-u^2)^p\int_{u\D}|F'(z)|^p\,dm(z).
\end{align*}
Then for $1/2<u<1$,
$$
A_p(u, F')=\left(\frac{1}{u^2}\int_{u\D}|F'(z)|^p\,dm(z)\right)^{1/p}\leq \frac{C'}{1-u}\n{F}_{A^p}.
$$
On the other hand there is a constant $C''$ such that
$$
|F'(z)|\leq C''\n{F}_{A^p},
$$
for  $F\in A^p,$ and $|z|\leq 1/2$,
\cite[page 99, exer. 24] {Zhu}. Therefore $A_p(u, F')\leq C''\n{F}_{A^p}$ when $0\leq u\leq 1/2$.
Choosing $C=\max\{C', C''\}$ gives (\ref{Be3}).

Using (\ref{Be3}) and observing that if $f\in A^p$ then $\Phi_{[s]}\in A^p$ we have
\begin{align*}
\frac{1}{C}(1-u)A_p(u, \Phi_{[s]}')&\leq\n{\Phi_{[s]}}_{A^p}=\n{f-\frac{1}{s}f_s}_{A^p}\\
&\leq \n{f-f_s}_{A^p}+\n{f_s-\frac{1}{s}f_s}_{A^p}\\
& =\n{f_s-f}_{A^p}+\frac{1-s}{s}\n{f_s}_{A^p}\\
&\leq \frac{1}{r}(1-s)\n{f}_{A^p}+\n{f_s-f}_{A^p}
\end{align*}
for $r\leq s<1$ and  $0\leq u<1$. The proof can now be finished as in the case of Hardy spaces. Namely choose $u=r$ and integrate to obtain
$$
\int_r^1A_p(r, \Phi_{[s]}')\,ds\leq \frac{C\n{f}_{A^p}}{2r}(1-r)+ \frac{C}{1-r}\int_r^1\n{f_s-f}_{A^p}\,ds.
$$
Then use  (\ref{Be1}) with $u=r$  and  (\ref{Be2})  to  find
\begin{equation}\label{Be4}
\begin{aligned}
(1-r)A_p(r, f')&\leq C_p\n{f_r-f}_{A^p}+\frac{C\n{f}_{A^p}}{2r}(1-r)\\
&+ \frac{C}{1-r}\int_r^1\n{f_s-f}_{A^p}\,ds.
\end{aligned}
\end{equation}
Using the assumption   $\n{f_r-f}_{A^p}\leq C'(1-r)^{\a}$ we have
$$
(1-r)A_p(r, f')\leq C_p' (1-r)^{\a}+\frac{C\n{f}_{A^p}}{2r}(1-r)+\frac{C C'}{\a+1}(1-r)^{\a}
$$
therefore
$$
A_p(r, f')=\og((1-r)^{\a-1}), \quad r\to 1^{-},
$$
and this completes the proof.
\end{proof}

\textbf{The Dirichlet space}.  We next show that the analogue of theorem (\ref{theorem3}) in valid in the Dirichlet space $\mathcal{D}$.  This is in fact a corollary of theorem (\ref{theorem3}) so we will be brief.
Recall that $\mathcal{D}$ contains those analytic  $f$ such that $f'\in A^2$. It is a Hilbert space with the norm
$$
\n{f}_{\mathcal{D}}= (|f(0)|^2+\n{f'}_{A^2})^{1/2}.
$$

For $f\in \mathcal{D}$ the dilations $f_r$, $r<1$,  are in $\mathcal{D}$. We define the quantity
$$
D(r, f)=\n{f_r}_{\mathcal{D}}, \quad 0<r<1.
$$
Further for  $w\in \overline{\D}$  write  $f_w(z)=f(wz)$. Using the triangle inequality we obtain
$$
\n{f_w-f}_{\mathcal{D}}\leq |1-w|\n{f}_{\mathcal{D}}+\n{(f')_w-f'}_{A^2},
$$
and
$$
\n{(f')_w-f'}_{A^2}\leq \frac{|1-w|}{|w|}\n{f}_{\mathcal{D}}+\n{f_w-f}_{\mathcal{D}}.
$$
In particular letting $w=e^{it}$ we see that for $0<\a\leq 1$, each of the conditions     $\n{r_t(f)-f}_{\mathcal{D}}=\og(|t|^{\a})$ and  $\n{r_t(f')-f'}_{A^2}=\og(|t|^{\a})$ implies the other.

Similarly letting $w=r\in (0,1]$ we find that the two conditions $\n{f_r-f}_{\mathcal{D}}=\og((1-r)^{\a}) $ and
$\n{(f')_r-f'}_{A^2}=\og((1-r)^{\a}) $ are equivalent.

Collecting all these observations we obtain as a corollary of Theorem (\ref{theorem3}) the following analogue on  $\mathcal{D}$.

\begin{theorem}
 Let $0<\a\leq 1$ and $f\in \mathcal{D}$. Then the following are equivalent
\begin{itemize}
\item[(a)] $\n{r_t(f)-f}_{\mathcal{D}}=\og(|t|^{\a}), \quad t\to 0$,
\item[(b)] $D(r, f')  =\og\left ((1-r)^{\a-1}\right ), \quad r\to 1^{-}$,
\item[(c)] $\n{f_r-f}_{\mathcal{D}}=\og((1-r)^{\a}), \quad r\to 1^{-}$.
\end{itemize}
\end{theorem}

\section{Concluding Remarks}

The classical holomorphic Lipschitz spaces $\Lambda_{\a}(\D)$, $0<\a\leq 1$,  are defined to contain all analytic functions on $\D$ such that
$$
|f(z)-f(w)|\leq C|z-w|^{\a}, \quad z,w\in D.
$$
Such functions are continuous on the close disc so  $\Lambda_{\a}(\D)$ are contained in the  disc algebra $\mathcal{A}$. For these spaces it was proved already by Hardy and Littlewood that $ f\in \Lambda_{\a}(\D)$ if and only if
$$
M_{\infty}(r, f')=\og((1-r)^{\a-1}), \,\,r\to 1^{-},
$$
\cite[Theorem 5.1]{Du}.
If we
use the identities (\ref{Ef1}) and (\ref{Ef2}) and take the $\infty$-means $M_{\infty}(r, f)$ instead of the $p$-means
$M_p(r,f)$ we can obtain the analogue of Theorem C for the disc algebra, i.e. if $0<\a\leq 1$ and  $f$ is analytic on $\D$ and continuous on $\overline{\D}$ then the following conditions are equivalent
\begin{itemize}
\item[(i)] $M_{\infty}(r, f')=\og((1-r)^{\a-1}), \,\,r\to 1^{-} $
\item[(ii)] $\n{f_r-f}_{\infty}=\og((1-r)^{\a}), \,\,r\to 1^{-} $
\end{itemize}
This result is already known  in a much more general setting, see for example  \cite{Pa1} for a detailed description.

Notice that all available relevant results can be included in a single statement as follows:
Let $X$ denote any of the Banach spaces $H^p$,   $A^p$,   $\mathcal{D}$ or  $\mathcal{A}$,  and  $\n{\,\,}_X$ be the corresponding norm. Then if  $0<\a\leq 1$ and $f\in X$  the following are equivalent
\begin{itemize}
\item[(a)] $\n{r_t(f)-f}_{X}=\og(|t|^{\a}),\quad t\to 0$,
\item[(b)] $\n{(f')_r}_X  =\og\left (1-r)^{\a-1}\right ), \quad r\to 1^{-}$,
\item[(c)] $\n{f_r-f}_X=\og((1-r)^{\a}), \quad r\to 1^{-}$.
\end{itemize}

\bibliographystyle{amsalpha}

\begin{thebibliography}{AAA}

\bibitem[BS]{BS} O. Blasco and G. S. de Souza,
Spaces of analytic functions on the disc where the growth of $M_ p(F, r)$ depends on a weight,
J. Math. Anal. Appl. 147 (1990),  580--598.


\bibitem[BSS]{BSS} P. S. Bourdon, J. H.Shapiro, W. T. Sledd, Fourier Series, Mean Lipschitz spaces and Bounded Mean Oscillation,  Analysis in Urbana Vol. 1,  LMS Lect. Note Ser. 137, (1989) 81--110.


\bibitem [Du]{Du} P. L. Duren,
\textit{Theory of $H^p$ spaces}, Academic Press, New York and London
1970.



\bibitem[HL1]{HL1} G. H. Hardy and J. E. Littlewood, Some properties of fractional integrals I.
Math. Z. 27 (1928),  565--606
.
\bibitem[HL2]{HL2} G. H. Hardy and J. E. Littlewood,
A convergence criterion for Fourier series, Math. Z. 28 (1928), 612--634.

\bibitem[HL3]{HL3} G. H. Hardy and J. E. Littlewood, Some properties of fractional integrals II, Math. Z.
34 (1932), 403--439.



\bibitem[Gi1]{Gi1} D. Girela, Mean Lipschitz spaces and Bounded Mean Oscillation, Illinois J. Math. 41 (1997), 214--230.

\bibitem[Gi2]{Gi2} D. Girela, Analytic Functions of Bounded Mean Oscillation, in Complex function spaces, Proc. summer school, Mekrijarvi (R. Aulaskari, ed.)  Rep. Ser., Dept. Math. Univ. Joensuu. 4, (2001) 61--170.

\bibitem[GiG]{GiG} D. Girela and C. Gonzalez, Some results on mean Lipschitz spaces of analytic functions,
Rocky Mt. J. Math. 30, (2000), 901--922.

\bibitem[JP]{JP} M. Jevtic and M. Pavlovic,
Besov-Lipschitz and mean Besov-Lipschitz spaces of holomorphic functions on the unit ball,
Potential Anal. 38 (2013), 1187--1206.



\bibitem[Pa1]{Pa1} M. Pavlovic, Lipschitz conditions on the modulus
of a harmonic function, Rev. Mat. Iberoamericana 23 (2007),  831–-845.

\bibitem[Pa2]{Pa2} M. Pavlovic, Function classes on the unit disc,  An introduction.
de Gruyter Studies in Math. Vol. 52, de Gruyter,  Berlin, 2014.

\bibitem[Pa3]{Pa3} M. Pavlovic, On the moduli of continuity of $H^p$ functions with $0<p<1$,
Proc. Edinb. Math. Soc. 35 (1992), 89--100,


\bibitem[HS]{HS} H. S. Shapiro,
The modulus of continuity of an analytic function, in
Spaces of Analytic Functions, Edited by  O. B. Bekken,
 B. K. Oksendal,
 A. Stray, Lecture Notes in Mathematics  512,  Springer 1976, pp 131--138


\bibitem[ES]{ES} Eh. A. Storozhenko, On a problem of Hardy-Littlewood, Math. USSR, Sb. 47 (1984) 557--577; translation
from Mat. Sb., Nov. Ser. 119 (161) (1982), 564--583.

\bibitem[Wa]{Wa} D. Walsh, Criteria for Membership of the
Mean Lipschitz Spaces, Z. Anal. Anwend. 22 (2003),  339 -- 355.

\bibitem[XZ]{XZ} J. Xiao and K. Zhu,  Volume integral means of holomorphic functions, Proc. Amer. Math. Soc. 139 (2011), 1455--1465.
\bibitem[Zhu]{Zhu} K. Zhu, Operator Theory in Function Spaces, 2nd edn. Math Surveys and Monographs vol. 138, AMS,
Providence, RI, 2007.



\end{thebibliography}

\end{document}